\documentclass[11pt]{article}
\usepackage{amssymb}
\usepackage{latexsym}
\usepackage{amsthm}
\usepackage{amscd}
\usepackage{amsmath}
\usepackage{diagrams}
\usepackage{indentfirst}
\usepackage[pdftex]{graphicx}

\theoremstyle{definition}
\newtheorem{thm}{Theorem}[section]
\newtheorem{lem}[thm]{Lemma}
\newtheorem{th-def}[thm]{Theorem-Definition}
\newtheorem{cor}[thm]{Corollary}
\newtheorem{defn-lem}[thm]{Definition-Lemma}

\newtheorem{rem}[thm]{Remark}

\numberwithin{equation}{section}

\allowdisplaybreaks[4]

\def \Q{{\mathbb Q}}

\def \C{{\mathbb C}}
\def \F{{\mathbb F}}

\def \Z{{\mathbb Z}}
\def \R{{\mathbb R}}

\def\map#1.#2.{#1 \longrightarrow #2}
\def\rmap#1.#2.{#1 \dasharrow #2}
\DeclareMathOperator{\rank}{rank}

\DeclareMathOperator{\ord}{ord}

\DeclareMathOperator{\im}{Im}

\DeclareMathOperator{\Aut}{Aut}

\def\fb#1.{\underset #1 \to \times}
\def\pr#1.{\Bbb P^{#1}}
\def\ring#1.{\mathcal O_{#1}}
\def\mlist#1.#2.{{#1}_1,{#1}_2,\dots,{#1}_{#2}}

\def\uloopr#1{\ar@'{@+{[0,0]+(-4,5)} @+{[0,0]+(0,10)}
@+{[0,0]+(4,5)}}
  ^{#1}}
\def\dloopr#1{\ar@'{@+{[0,0]+(-4,-5)} @+{[0,0]+(0,-10)}
@+{[0,0]+(4,-5)}}
  _{#1}}

\def\rloopd#1{\ar@'{@+{[0,0]+(5,4)} @+{[0,0]+(10,0)}
@+{[0,0]+(5,-4)}}
  ^{#1}}
\def\lloopd#1{\ar@'{@+{[0,0]+(-5,4)} @+{[0,0]+(-10,0)}
@+{[0,0]+(-5,-4)}}
  _{#1}}

\long\def\ignore#1{}
\long\def\ignore#1{#1}
\title{The non-symplectic index of supersingular K3 surfaces}
\author{Junmyeong Jang}
\date{}
\begin{document}

\newpage
 \normalsize
\maketitle
\vspace{0.3cm}

\small

\vspace{0.3cm}

\normalsize
\medskip
     \section{Introduction}
       Let $X$ be an algebraic complex K3 surface. The second integral singular cohomology $H^{2}(X/\Z)$ is an even unimodular integral lattice of signature $(3,19)$. This lattice is unique and isomorphic to $U ^{3} \oplus E_{8}^{2}$. Here $U$ is the unimodular hyperbolic lattice and $E_{8}$ is the negative definite root lattice of type $E_{8}$. The Neron-Severi group of $X$ is an even lattice of signature $(1, \rho(X) -1)$. Here the rank of $NS(X)$, $\rho (X)$ is the Picard number of $X$. The cycle map induces a primitive embedding of lattices
     $$NS(X) \hookrightarrow H^{2}(X/\Z).$$
     The orthogonal complement of this embedding is the transcendental lattice of $X$ which is denoted by $T(X)$.
     $T(X)$ is an even integral lattice of signature $(2, 20 -\rho (X))$. By the Hodge decomposition, we may regard the one dimensional complex space of
     the global two forms of $X$, $H^{0}(X, \Omega ^{2}_{X/\C})$ is a direct factor of $T(X) \otimes \C$. Let us denote the representations of the automorphism group of $X$, $\Aut(X)$ on $T(X)$ and $H^{0}(X, \Omega _{X/\C} ^{2})$ by
     \begin{center}
     $\chi _{X} : \Aut(X) \to O(T(X))$ and $\rho _{X} : \Aut(X) \to GL(H^{0}(X,\Omega ^{2}_{X/\C}))$.
     \end{center}
     Then there exists a projection $p_{X} : \im  \chi _{X} \to \im \rho _{X}$. It is known that $p_{X}$ is an isomorphism and $\im \chi _{X} \simeq \im \rho _{X}$ is a finite cyclic group. (\cite{Ni1}) In particular, for any $\alpha \in \Aut(X)$, $\chi _{X}(\alpha)$ is of finite order. Moreover, if $\ord \chi_{X}(\alpha) = n$, by the Lefschetz (1,1) theorem, every eigenvalue of $\chi _{X}(\alpha)$ is a primitive $n$-th root of unity. We denote the characteristic polynomial in a variable $T$ of a linear operator $L$ by $\varphi (L)$. Let $\Phi _{n}(T)$ be the $n$-th cyclotomic polynomial. If $\ord \chi _{X}(\alpha) =n$, since $\varphi (\chi _{X}(\alpha))$ is an integral polynomial, $\varphi (\chi _{X}(\alpha))$ is a power of $\Phi _{n}(T)$. Therefore, when $N =| \im \rho _{X}|$, the rank of $T(X)$ is a multiple of $\phi (N)$. Here $\phi$ is the Euler $\phi$ function. We say $N$ is the non-symplectic index of $X$. Because $\phi (N) \leq \rank T(X)$ and $\rank T(X)$ is at most 21, $\phi (N) \leq 20$. There exists a complex K3 surface of non-symplectic index 66, which is the maximum. (\cite{Vo}, \cite{Ko})  \\

     Assume that $k$ is an algebraically closed filed of odd characteristic $p$, $W$ is the ring of Witt vectors of $k$ and $K$ is the fraction field of $W$.\\

 Assume $X$ is a K3 surface defined over $k$.  
The second crystalline cohomology of $X$, $H^{2}_{cris}(X/W)$ is a free $W$ module of rank 22 equipped with a canonical Frobenius semi-linear endomorphism $\mathbf{F} : H^{2}_{cris}(X/W) \to H^{2}_{cris}(X/W)$. There exists a cycle map 
$$NS(X) \otimes W \to H^{2}_{cris}(X/W)$$
which preserves the lattice structure. 
Assume  $X$ is of finite height $h$. $(1 \leq h \leq 10)$ The Picard number of $X$ is at most $22-2h$. 
The orthogonal complement of the cycle map is called the crystalline transcendental lattice and it is denoted by $T_{cris}(X)$.
There exists an algebraic lifting of $X$ over $W$, 
$\mathcal{X}/W$ satisfying  the reduction map 
$$NS(\mathcal{X}\otimes K) \to NS(X)$$
is an isomorphism. (\cite{NO}, \cite{LM}, \cite{J1})
We call such a lifting a Nerno-Severi group preserving lifting.
Let us denote $\mathcal{X}_{K} = \mathcal{X} \otimes K$ and $\mathcal{X}_{\bar{K}} = \mathcal{X} \otimes \bar{K}$.
For a Neron-Severi group preserving lifting $\mathcal{X}/W$, $NS(\mathcal{X} _{\bar{K}}) = NS(\mathcal{X}_{K}) = NS(X)$
and every automorphism of $\mathcal{X}_{\bar{K}}$ is extendable to the integral model $\mathcal{X} \otimes \mathfrak{o}_{\bar{K}}$.
It follows that the reduction map 
$$\Aut(\mathcal{X}_{\bar{K}}) \to \Aut(X)$$
is injective and its image is of finite index in $\Aut(X)$.
Moreover, by the comparison theorem, we can identity $T(\mathcal{X}_{\bar{K}})$ with $T_{cris}(X)$. (\cite{BO1}, Corollary 3.7) Therefore the image of the representation 
$$\chi _{cris,X} : \Aut (X) \to O(T_{cris}(X))$$
is finite. For a K3 surface of finite height $X$, $H^{2}(X,W\mathcal{O}_{X})$ is a direct factor of $T_{cris}(X)$ and there is a projection
$H^{2}(X,W\mathcal{O}_{X}) \twoheadrightarrow H^{2}(X,\mathcal{O}_{X})$. Considering the Serre duality, we have the canonical projection
$$ \im \chi _{cris,X} \to \im \rho _{X}$$
and the non-symplectic index of $X$ is finite. Assume $N$ is the non-symplectic index of $X$ and $\rho _{X}(\alpha)$ generates $\im \rho _{X}$ for $\alpha \in \Aut(X)$. Assume $\ord \chi _{cris, X}(\alpha) = p^{a} \cdot m$, where $m$ is relatively prime to $p$. Let $\beta = \alpha ^{p^{a}}$. Then $\ord \chi _{cris, X}(\beta) = m$ and $\rho (\beta)$ also generates $\im \rho _{X}$. Since  $\ord \chi _{cris, X}(\beta)$ is relatively prime to $p$, there is an algebraic lifting $\mathcal{X}'/W$ of $X$ such that $\beta$ is liftable to $\mathcal{X}'$. (\cite{J3}, Theorem 3.2)
Let $\mathfrak{b} : \mathcal{X}' \to \mathcal{X}'$ be the lifting of $\beta$. 
The order of $\rho _{\mathcal{X}'}(\mathfrak{b})$ is $N$ and
the characteristic polynomial $\varphi (\chi _{\mathcal{X}' _{\bar{K}}} (\mathfrak{b}) ) = \varphi ( \chi _{cris,X}( \beta))$ has integer coefficients and is a power of $N$-th cyclotomic polynomial $\Phi _{N}(T)$. Therefore $\phi (N)$ divides $\rank T_{cris}(X)$ and $\phi (N) \leq 20$. The non-symplectic index of a K3 surface of finite height over $k$ is at most 66.\\

Now assume $X$ is a supersingular K3 surface over $k$. The Picard number of $X$ is 22. Let us denote the dual lattice of a lattice $L$ by $L^{*}$ and  
the discriminant group of $N$, $N^{*}/N$ by $l(N)$. 
The discriminant group $l(NS(X)) = (\Z/p)^{2\sigma}$ for an integer $\sigma$ between 1 and 10. Here $\sigma$ is called the Artin invariant of $X$. All the supersingular K3 surfaces of Artin invariant $\sigma$ form a $\sigma -1$ dimensional family and a supersingular K3 surface of Artin invariant 1 is unique up to isomorphism. The Neron-Severi lattice a supersingular K3 surface is determined by the base characteristic $p$ and the Artin invariant. 
We denote a general supersingualr K3 surface of Artin invariant $\sigma$ defined over a field of characteristic $p$ by $X_{p,\sigma}$.
 We also denote $NS(X_{p,\sigma})$ by $N_{p,\sigma}$.
 It is known that the non-symplectic index of $X _{p,\sigma}$ is a divisor of $p^{\sigma} +1$. (\cite{Ny}) In a previous work, we prove that the following. (\cite{Jc}, Theorem 3.3)\\

     \textbf{Theorem} If $p>3$, the non-symplectic index of a supersingular K3 surface of Artin invarinat 1 is $p+1$.\\

As a direct corollary, we have the following. (Loc. cit)\\

 \textbf{Corollary} If $\phi (p+1) > 20$, then $X_{p,1}$ has an automorphism which can not lifted over a ring of characteristic 0.\\

This problem of existence of a non-liftable automorphism for supersingular K3 surfaces is completely answered. (\cite{Sh}, \cite{Yu}, \cite{Br})\\

 \textbf{Theorem} Every supersingular K3 surface has a non-liftable automorphism. \\

This result is based on the study of the Salem degree of an automorphism of a K3 surface.\\

In this paper, we generalize our previous result
to obtain the non-symplecitc indexes of all supersingular K3 surfaces when the base characteristic $p>3$.

\section{Strictly characteristic subspaces and Crystalline Torelli theorem}
 In this section, we review Ogus' classification of supersingular K3 surfaces and the crystalline Torelli theorem in \cite{Og1} and \cite{Og2}.\\

Assume the characteristic of $k$, $p>3$ and $X$ is a supersingular K3 surface of Artin invariant $\sigma$ over $k$. The discriminant group $l(NS(X))$ is a $2\sigma$ dimensional vector space over the prime field $\F _{p}$ equipped with the induced quadratic form. The discriminant of this quadratic form is $(-1)^{\sigma} \Delta$ where $\Delta$ is a quadratic non-residue modulo $p$. Hence $l(NS(X))$ does not have a $\sigma$ dimensional isotropic subspace over $\F_{p}$. The cycle map $NS(X) \otimes W \hookrightarrow H^{2}_{cris}(X/W)$ gives the following chain
$$NS(X) \otimes W \subset H^{2}_{cris}(X/W) \subset NS(X)^{*} \otimes W.$$
The cokernel $H^{2}_{cris}(X/W)/(NS(X) \otimes W)$ is a $\sigma$ dimensional isotropic $k$ subspace of $(NS(X)^{*} \otimes W)/(NS(X) \otimes W) = l(NS(X)) \otimes k$. We denote $H^{2}_{cris}(X/W)/ (NS(X) \otimes W)$ by $\mathcal{K}(X)$. Note that $\mathcal{K}(X)$ determines the inclusion
$NS(X) \otimes W \subset H^{2}_{cris}(X/W)$.
We set $f = id \otimes F_{k}: l(NS(X)) \otimes k \to l(NS(X)) \otimes k$. A $\sigma$ dimensional isotropic subspace of $l(N_{p,\sigma}) \otimes k$, $\mathcal{K}$ is a characteristic subspace if $\mathcal{K} + f (\mathcal{K})$ is $\sigma +1$ dimensional. A characteristic subspace $\mathcal{K}$ is a strictly characteristic subspace if there exists no proper $\F_{p}$ subspace $M$ of $l(N_{p, \sigma})$ such that $\mathcal{K} \subset M \otimes k$. 

\begin{thm}[\cite{Og1}, Theorem 3.20]
For a supersingular K3 surface $X$, $\mathcal{K}(X)$ is a strictly characteristic subspace of $l(NS(X)) \otimes k$.
\end{thm}
\begin{rem}
Our definition of $\mathcal{K}(X)$ above is slightly different from one given in \cite{Og1}. But all the argument in this paper work for both definitions. 
\end{rem}
Assume $\mathcal{K}$ is a strictly characteristic subspace of $l(N_{p, \sigma}) \otimes k$. Then $l_{\mathcal{K}} = \cap _{i=0} ^{\sigma -1} f^{i} \mathcal{K}$ is a line in $l(N_{p,\sigma}) \otimes k$. Moreover we have 
\begin{center}
$\sum _{i=0}^{\sigma -1} f ^{-i} (l_{\mathcal{K}}) = \mathcal{K}$ and 
$\sum _{i=0}^{2\sigma -1} f ^{-i} (l_{\mathcal{K}}) = l(N_{p, \sigma}) \otimes k$.
\end{center} 
 Let $v$ be a non zero vector in $l_{\mathcal{K}}$ and we denote $v_{i} = f ^{1-i} (v)$. Then $\{ v_{1} , \cdots , v_{\sigma} \}$ is a basis of 
$\mathcal{K}$ and $\{ v_{1} , \cdots , v_{2\sigma} \}$ is a basis of $l(N_{p,\sigma}) \otimes k$. Because $l(N_{p,\sigma})$ is non-degenerated, $v_{1} \cdot v_{\sigma +1} \neq 0$. After a suitable scalar multiplication, we may assume $v_{1} \cdot v_{\sigma +1} =1$. Here $v_{1}$ is uniquely determined up to $(p^{\sigma}+1)$-th roots of unity multiplication. Now we put $a_{i} = v_{1} \cdot v_{\sigma +1+i} \in k$. $(\i = 1, \cdots , \sigma -1 )$
The intersection matrix of $l(N_{p,\sigma})$ in terms of the basis $ v_{1} , \cdots , v_{2\sigma}$ is $ \left (
\begin{array}{cc}
0 & A \\
A^{t} & 0 \\
\end{array} \right ) $, where
\begin{equation}\label{mat}
   A = \left (
\begin{array}{cccccc}
1 & a_{1} & a_{2} & a_{3} & \cdots & a_{\sigma -1}\\
0 & 1 & F_{k}^{-1}(a_{1}) & F_{k}^{-1}(a_{2}) & \cdots & F_{k}^{-1} (a_{\sigma -2})\\
0 & 0 & 1 & F_{k}^{-2} (a_{1}) & \cdots &  F_{k}^{-2}(a_{\sigma -3}) \\
\vdots & \vdots & \vdots & \vdots &  & \vdots \\
0 & 0 & 0 & 0 & \cdots & 1
\end{array}
 \right ). 
\end{equation}

If we replace $v_{1}$ by $\xi v_{1}$ $(\xi ^{p^{\sigma}+1} =1)$, then $a_{i}$ is replaced by $\xi ^{p^{\sigma +i}+1} a_{i}$. Hence we have a map 
$$\Psi : \mathcal{K} \mapsto (a_{1} , \cdots , a_{\sigma -1}) \in \mathbb{A} ^{\sigma -1} / \mu _{p^{\sigma}+1} (k) .$$

\begin{thm}[\cite{Og1}, Theorem 3.21] The map $\Psi$ is bijective from the set of isomorphic classes of $(l(N_{p, \sigma}), \mathcal{K})$ to $\mathbb{A} ^{\sigma -1} / \mu _{p^{\sigma}+1}(k)$.

\end{thm}
We put $\Delta _{X} = \{ v \in NS(X) | v \cdot v = -2 \}$. For any $v \in \Delta _{X}$, let 
   $$s _{v} : w \mapsto w + (v \cdot w) v \in O(NS(X))$$ 
be the reflection along the line of $v$. The Wyle group of $NS(X)$ is the subgroup of $O(NS(X))$ generated by $s_{v}$ $(v \in \Delta _{X})$ and $-id$. We denote the Wyle group of $X$ by $W_{X}$. The real quadratic space $NS(X) \otimes \R$ is a hyperboic space of rank $(1,21)$. The positive cone of $NS(X) \otimes \R$, $\mathcal{P}_{X} = \{ v \in NS(X) \otimes \R | v \cdot v >0 \}$ has two connected components. Let $\mathcal{C}_{X}$ be the set of connected components of $\mathcal{P}_{X} - \cup _{v \in \Delta _{X}} <v> ^{\bot} ( \subset NS(X) \otimes \R )$. The ample cone of $X$, $\mathcal{A}_{X}$ is an element of $\mathcal{C}_{X}$. The Wyle group $W_{X}$ acts simply and transitively on $\mathcal{C}_{X}$ by the canonical action. The automorphism group of $X$, $\Aut(X)$ acts naturally on $NS(X)$ and $l(NS(X))$. We denote these representations of $\Aut(X)$ by 
\begin{center}
$\lambda _{X} : \Aut (X) \to O(NS(X))$ and $\nu _{X} : \Aut(X) \to O(l(NS(X)))$.

\end{center}
For an autormophism $\alpha \in \Aut(X)$, $\lambda _{X}(\alpha)$ preserves $\mathcal{A}_{X}$ and $\nu (\alpha)$ preserves $\mathcal{K}(X)$.

\begin{thm}[\cite{Og2}, Theorem II, III, Crystalline Torelli theorem] ${ }$ \\
(1) If $g \in O(NS(X))$ preserves $\mathcal{A}_{X}$ and $\mathcal{K}(X)$, then there exists a unique $\alpha \in \Aut(X)$ such that $\lambda _{X}(\alpha) = g$.

(2) If $\mathcal{K}$ is a strictly characteristic space of $l(N_{p, \sigma})$, there exists a unique supersingular K3 surface $X$  of Artin invariant $\sigma$ up to isomorphism over $k$ such that $(l(NS(X)), \mathcal{K}(X))$ is isomorphic to $(l(N_{p, \sigma}) , \mathcal{K})$.
 
\end{thm}
By the theorem, we may regard the automorphism group $\Aut(X)$ as a subgroup of $O(NS(X))$.
Two representations of $\Aut(X)$, $\rho _{X}$ and $\nu _{X}$ are isomorphic (\cite{Ny}, \cite{J2}, \cite{BG}), so $\im \rho _{X} \simeq \im \nu _{X}$ is a finite cyclic group. Since, for any $\alpha \in \Aut(X)$, $\nu (\alpha)$ is rational over $\F_{p}$ and preserves $\mathcal{K}(X)$, $\nu (\alpha) (l_{\mathcal{K}(X)}) = l_{\mathcal{K}(X)}$. And $v_{1} , \cdots , v_{2\sigma}$ are all eigen vectors of $\nu (\alpha)$. The character of $\Aut(X)$ on $H^{0}(X, \Omega _{X/k} ^{2})$ is isomorphic to the character of $\Aut(X)$ on the 1 dimensional space $<v_{\sigma +1}>$. (\cite{Og1}) Hence the character of $\Aut(X)$ on $H^{2}(X, \mathcal{O}_{X})$ is isomorphic to the character of $\Aut(X)$ on $<v_{1}>$. If $\nu (\alpha)(v_{1}) = \xi v_{1}$, $F_{k}^{-\sigma}(\xi) = \xi ^{p^{-\sigma}} =\xi ^{-1}$ and $\nu (\alpha) (v_{\sigma +1}) = \xi ^{-1} v_{\sigma +1}$.  

\begin{lem}[c.f. \cite{BG}, Corollary 3.8]\label{lemm}
 Assume $g \in O (l(NS(X)))$ preserves $\mathcal{K}(X)$. Then there exists $\alpha \in \Aut(X)$ such that $\nu (\alpha) =g$.

\end{lem}
\begin{proof}
Because the signature of $NS(X)$ is $(1,21)$ and the length of $l(NS(X))$ is at most 20,
the class number of $NS(X)$ is 1 and the reduction map $O(NS(X)) \to O(l(NS(X)))$ is surjective. (\cite{Ni}, Theorem 1.14.2)
There exists $\mathfrak{g} \in O(NS(X))$ such that $\mathfrak{g} | l(NS(X)) =g$. Let $\mathfrak{g} (\mathcal{A}_{X}) = \mathcal{B}\in \mathcal{C}_{X}$. There also exists a unique $\varphi \in W_{X}$ such that $\varphi (\mathcal{B}) =\mathcal{A}_{X}$. 
If $\mathcal{B}$ is in the same connected component of $\mathcal{P}_{X}$ with $\mathcal{A}_{X}$, we can write
$\varphi = s _{w_{1}} \circ s_{w_{2}} \circ \cdots \circ s_{w_{n}}$ for $w_{i} \in \Delta _{X}$.
Let $M$ be the rank 1 sub lattice of $NS(X)$ generated by $w \in \Delta _{X}$. Then $l(M) = \Z/2$ and $l(M^{\bot}) = \Z/2 \oplus l(NS(X))$.
Since $s_{w} | M^{\bot} = id$, $s_{w}$ induces the identity map on $l(NS(X))$. It follows that $\varphi \circ \mathfrak{g}$ preserves $\mathcal{A}_{X}$ and $\mathcal{K}(X)$, so by the crystalline Torelli theorem, $\varphi \circ \mathfrak{g} \in \Aut(X)$. And $\nu(\varphi \circ \mathfrak{g}) = g$.\\

Now assume $\mathcal{B}$ is in the other connected component of $\mathcal{P}_{X}$ with $\mathcal{A}_{X}$. 
The signature of $NS(X)$ is $(1,21)$, there exists a vector in $NS(X) \otimes \R$ of self intersection 2. If $q$ is a prime number different from $p$, 
$NS(X) \otimes \Z_{q}$ is even unimodular of rank 22, so it contains a vector of self intersection 2. The length of the unimodular part of $NS(X) \otimes \Z_{p}$ is at least 2, so $NS(X) \otimes \Z_{p}$ has a vector of self intersection 2. By the Hasse principle, there exists a lattice in the genus of $NS(X)$ which contains a vector of self intersection 2. But the class number of $NS(X)$ is 1, so $NS(X)$ contains a vector $u$ such that $u \cdot u =2$. 
Let 
$$t_{u} : w \mapsto w - (w \cdot u)u \in O(NS(X))$$ 
be the reflection along the line of $u$. Then $t_{u}$ exchanges two connected components of $\mathcal{P}_{X}$ and $t_{u}| l(NS(X)) =id$. And $\mathcal{B}' = t_{u}(\mathcal{B})$ is in the same connected component of $\mathcal{P}_{X}$ with $\mathcal{A}_{X}$. If $\varphi (\mathcal{B}') = \mathcal{A}_{X}$ $(\varphi \in W_{X})$, $\varphi$ is the composition of reflections along -2 vectors. It follows that $\varphi \circ t_{u} \circ \mathfrak{g} \in \Aut(X)$ and
$\nu(\varphi \circ t_{u} \circ \mathfrak{g}) = g$.
\end{proof}

\begin{cor}
The non-symplectic index of $X$ is even.
\end{cor}
\begin{proof}
Since $-id \in O(l(NS(X))$ preserves every sub space of $l(NS(X)) \otimes k$, $-id \in \im \nu_{X}$ and the order of $-id$ is 2. 
\end{proof}

     \section{Main result }
Let $k$ be an algebraically closed field of chracteristic $p>3$ and $X$ be a supersingular K3 surface of Artin invariant $\sigma$ over $k$.
\begin{thm}
The non-symplectic index of $X$ over $k$ is as in the Table 1.
\end{thm}
\begin{center}
\begin{tabular}{|c|c|c|}
     \hline
     $\sigma$ & non-symplectic index & family \\
     \hline
     1 & $p+1$&  unique \\
     \hline
     2 & 2 & generic\\
     & $p^{2}+1$ & unique \\
     \hline
     3 & 2 & generic\\
     & $p+1$ & 1 dimensional\\
    &  $p^{3}+1$ & unique \\
     \hline
     4 & 2&  generic\\ & $p^{4}+1$ & unique\\
     \hline
     5 & 2 & generic \\ & $p+1$ & 2 dimensional \\ & $p^{5}+1$ & unique \\
     \hline
     6 & 2 & generic \\ & $p^{2}+1$ & 1 dimensional \\ &  $p^{6}+1$ &unique\\
     \hline
     7 & 2 & generic \\ & $p+1$ & 3 dimensional\\ & $p^{7} +1$ & unique\\
     \hline
     8 & 2 & generic \\ & $p^{8} +1$ & unique \\
     \hline
     9 & 2 & generic \\ & $p+1$ & 4 dimensional \\ & $p^{3} +1$ & 1 dimensional \\ &  $p^{9} +1$ & unique \\
     \hline
     10 & 2 & generic \\ &  $p^{2}+1$ & 2 dimensional \\ & $p^{10}+1$ & unique \\
     \hline
     \end{tabular}\\
\vspace{0.2cm}
Table 1
\end{center}
 
\begin{proof}
Let $v_{1}, \cdots , v_{2\sigma}$ be the basis of $l(NS(X)) \otimes k$ in the section 2 which makes the intersection matrix (\ref{mat}). For all $\alpha \in \Aut(X)$, each $v_{i}$ is an eigenvector of $\nu _{X}(\alpha)$. Assume $\nu (\alpha)(v_{1}) = \xi v_{1}$ and the order of $\xi$ in $k^{*}$ is $n$. 
Then $\nu (\alpha) (v_{i}) = F_{k}^{1-i} (\xi) v_{i}$ and $\xi$ determines $\nu (\alpha)$. Let $m$ be the smallest non-negative integer such that
$F_{k}^{-m}(\xi) = \xi ^{-1}$.
 If $m=0$, $\xi$ is 1 or -1  and there is no restriction for $a_{i} = v_{1} \cdot v_{\sigma +1+i}$. $(i =1 ,\cdots , \sigma -1$) 
If $m >0$,
$m$ is a divisor of  $\sigma$ and $\sigma / m$ is an odd integer. The order of $p$ in $(\Z /n)^{*}$ is $2m$ and $n$ is a divisor of $p^{m}+1$. Becuase $\nu (\alpha) \in O(l(NS(X)) \otimes k)$, $a_{i} =0$ unless $i$ is a multiple of $2m$.\\

Conversely, assume $m$ is a divisor of $\sigma$ such that $\sigma / m$ is an odd integer. Assume $n$ is a divisor of $p^{m}+1$ and $\xi \in k^{*}$ is a primitive $n$-th root of unity. Suppose $a_{i} =0$ unless $i$ is a multiple of $2m$. Let $g$ be a linear operator of $l(NS(X)) \otimes k$
which sends
$v_{i}$ to $F_{k}^{1-i}(\xi) v_{i}$. $(i=1, \cdots, 2\sigma)$ It is clear that $g$ preserves the subspace $\mathcal{K}(X)$. Since $a_{i} =0$ for $2m \nmid i$, $g \in O(l(NS(X)) \otimes k)$. $g$ is rational over $\F_{p}$ if and only if $f ^{-1}(g(w)) = g(f ^{-1}(w))$ for all $w \in l(NS(X)) \otimes k$. For that,
it is enough to check that $f ^{-1}(g(v_{i})) = g(f ^{-1}(v_{i}))$ for $i= 1 , \cdots , 2\sigma$. Because $f ^{-1}(v_{i})= v_{i+1}$ $(i=1, \cdots, 2\sigma -1)$, it is easy to see that this is true for $i = 1, \cdots, 2\sigma -1$. We put $v' = f ^{-1}(v_{2\sigma})$
and $v' = b_{1}v_{1} + b_{2} v_{2} + \cdots b_{2\sigma}v_{2\sigma}$. $(b_{i} \in k)$
By the assumption, $v_{2\sigma} \cdot v _{i} =0$ unless $i = \sigma -2me$ for a non-negative integer $e$ and $v_{2\sigma} \cdot v_{\sigma} =1$.
Since the pairing of $l(NS(X)) \otimes k$ is defined over $\mathbb{F} _{p}$,
for any $u,w \in l(NS(X)) \otimes k$, $F_{k}^{-1}(u \cdot w) = f ^{-1}(u) \cdot f ^{-1}(w)$. Hence 
$v' \cdot v_{i} = 0$  for $i \neq \sigma +1 - 2me$ $(i \geq 2)$ and $v' \cdot v' =0$. It is not difficult to see
that $b_{i} = 0$ unless $i =1+2me$ and $g(v') = \xi v'$, so $f ^{-1}(g(v_{2\sigma})) = g(f ^{-1}(v_{2\sigma}))$. 
It follows that $g \in O(l(NS(X))$ and, by Lemma \ref{lemm}, $g \in \im \nu _{X}$. Then the non-symplectic index of 
$X$ is divisible by $p^{m}+1$ and the proof is complete.
\end{proof}
\begin{rem} Simon Brandhorst proves, in his dissertation, partial result of Theorem 3.1. His proof is also based on the crystalline Torelli theorem.

\end{rem}
Let us say a unique supersingualr K3 surface of Artin invariant $\sigma$ whose non-symplectic index is $p^{\sigma} +1$ is the special supersingular 
K3 surface of Artin invariant $\sigma$. Every special supersingular K3 surface has a model over a finite field.
\begin{cor}
The maximal value of the non-symplectic index of a K3 surface defined over $k$ is $p^{10}+1$. The only K3 surface with the maximal non-symplectic index is the special supersingular K3 surface of Artin invariant 10. 
\end{cor}

\section{Example}
Let $X$ be a complex algebraic K3 surface and $N$ be the non-symplectic index of $X$. The rank of the transcendental lattice of $X$, $T(X)$ is divisible by 
$\phi (N)$. If the rank of $T(X)$ is equal to $\phi (N)$, $X$ is a CM K3 surface and $X$ has a model over a number field. There are many examples satisfying this condition. (See \cite{Br0})
Assume the rank of $T(X)$ is equal to $\phi (N)$ and $X$ is defined over a number field $F$. We choose $\alpha \in \Aut(X)$ such that the order of $\rho _{X}(g)$ is $N$. For almost all finite places $\upsilon$ of $F$, $X$ has a good reduction $X_{\upsilon}$ over $\upsilon$ and $g$ is expandable to $X_{\upsilon}$. For such a $\upsilon$, the height (and the Artin invariant) of $X_{\upsilon}$ is determined by the congruence class of the residue characteristic $p_{\upsilon}$ modulo $N$.
 In particular, if $m$ is the smallest positive integer such that $p_{\upsilon} ^{m} \equiv -1$ modulo $N$, then 
$X_{\upsilon}$ is a supersingular of Artin invariant $m$.  (\cite{J2}, Theorem 4.7, \cite{Je}, Theorem 2.3)
 Moreover, in this case, $\nu _{X_{\upsilon}} \in O(l(NS(X_{\upsilon}))) $ has $2\sigma$ distinct eigenvalues, so $a_{i} =0$ for all $i$ in (\ref{mat}) and $X_{\upsilon}$ is a special supersingular K3 surface.\\

For example, an elliptic K3 surface
$$X : y^{2} = x^{3}+t^{7}x+t$$
has a purely non-symplectic automorphism of order 38,
$$\alpha : (x,y,t) \mapsto (\xi ^{14}x, \xi ^{21}y, \xi^{4}t).$$
Here $\xi$ is a primitive 38th root of unity. If a rational prime $p$ does not divide 38, $(X, \alpha)$ has a good reduction $X_{p}$ over $\Z/p$.
If $p \equiv 3, 13, 15, 19, 29$ or 33 modulo 38 (and $p>3$), $X_{p}$ is special supersingular of Artin invariant 9.
If $p \equiv 27$ or 31 modulo 38, $X_{p}$ is special supersingular of Artin invariant 3.
If $p \equiv 37$  modulo 38, $X_{p}$ is supersingular of Artin invariant 1.\\

Because many complex K3 surfaces, for which the rank of transcendental lattice is equal to the phi value of the non-symplectic index, are define over $\Q$,
many special supersingular K3 surfaces have models over the prime field $\F_{p}$. We have the following question naturally.\\

{\large \textbf{Qestion.}}
 Is every special supersingular K3 surface has a model over $\F _{p}$?

        \vspace{0.6cm}
    {\bf Acknowledgment}\\
                           This research was supported by Basic Science Research Program through the National
Research Foundation of Korea (NRF) funded by the Ministry of Education, Science and
Technology [2015R1D1A1A01058962].\\

\vskip 1cm

\noindent
\newpage
J.Jang\\
Department of Mathematics\\
University of Ulsan \\
Daehakro 93, Namgu Ulsan 680-749, Korea\\ \\
jmjang@ulsan.ac.kr

     \end{document}